\documentclass{amsart}
\usepackage[T1]{fontenc}
\usepackage[utf8]{inputenc}

\usepackage[nice]{nicefrac}
\usepackage{calc}
\usepackage{amsmath}
\usepackage{amsfonts}
\usepackage{amssymb}
\usepackage{graphicx}
\usepackage{pst-all}
\usepackage{mathdots}

\usepackage{mathrsfs}

\DeclareFontFamily{OT1}{pzc}{}
\DeclareFontShape{OT1}{pzc}{m}{it}{<-> s * [1.050] pzcmi7t}{}
\DeclareMathAlphabet{\mathpzc}{OT1}{pzc}{m}{it}

\DeclareSymbolFont{EulerScript}{U}{eus}{m}{n} 
\SetSymbolFont{EulerScript}{bold}{U}{eus}{b}{n} 
\DeclareSymbolFontAlphabet\matheul{EulerScript}
 
\DeclareMathAlphabet{\mathant}{OMS}{antt}{m}{n} 
\DeclareMathAlphabet{\mathscrr}{OMS}{mdbch}{m}{n}

\newlanguage\fakelanguage
\newcommand\cyr{\fontencoding{OT2}\fontfamily{wncyr}\selectfont\language\fakelanguage}
\DeclareTextFontCommand{\textcyr}{\cyr}

\def\THM{\begin{theorem}} \def\thm{\end{theorem}}
\def\COR{\begin{corollario}} \def\cor{\end{corollario}}
\def\PRO{\begin{proposition}} \def\pro{\end{proposition}} 
\def\OBS{\begin{oss}} \def\obs{\end{oss}}
\def\DFN{\begin{definition}} \def\dfn{\end{definition}}
\def\LEM{\begin{lemma}} \def\lem{\end{lemma}}
\def\PR{\begin{proof}} \def\pr{\end{proof}}
\newcommand{\EX}{\begin{esercizio}} \newcommand{\ex}{\end{esercizio}}
\newcommand{\ES}{\begin{esempio}} \newcommand{\es}{\end{esempio}}

\renewcommand{\i}{\mb{i}}

\renewcommand{\bar}{\overline}
\renewcommand{\tilde}{\widetilde}

\newcommand{\lista}{\begin{itemize}} \newcommand{\ei}{\end{itemize}}

\def\({\left(}	\def\){\right)}
\def\[{\left[}	\def\]{\right]}
 
\def\<{\left<} \def\>{\right>}

\newcommand{\ba}[1]{\begin{array}{#1}}  \newcommand{\ea}{\end{array}}

\def\f#1#2{\frac{#1}{#2}}
\def\matrix#1#2{\left(\begin{array}{#1} #2 \end{array}\right)}

\def\3{\|\hskip-1pt |}

\newcommand{\Span}{\mathop{\text{Span}}}
\newcommand{\diag}{\mathop{\text{diag}}}

\DeclareMathOperator{\rank}{\text{rank}}

\def\_#1{\scriptstyle{#1}}

\def\t{\text{\itshape{\textsf{T}}}}

  \def\F{{\matheul F}} 
\def\B{{\matheul B}}

 \def\NN{{\mathbb{N}}} 
  
\def\ZZ{{\mathbb Z}}  
\def\CC{{\mathbb C}}  
\def\RR{{\mathbb R}}  
 \def\QQ{{\mathbb Q}}

     \def\ep{{\varepsilon}}

\def\sse{\Leftrightarrow}
\def\iff{\Longleftrightarrow}
\def\then{\Rightarrow}
\def\to{\rightarrow}
\def\Then{\Longrightarrow}

\def\conv#1{\xrightarrow{#1}}

\newtheoremstyle{teorema}{8pt}{8pt}{}{}{\bfseries}{}{8pt}{}
\theoremstyle{teorema}
\newtheorem{theorem}{Theorem}[section]
\newtheorem{proposition}[theorem]{Proposition}

\newtheorem{lemma}[theorem]{Lemma}
\newtheorem{corollary}[theorem]{Corollary}

\newtheorem{example}[theorem]{Example}

\newtheoremstyle{definizione}{8pt}{8pt}{\itshape}{}{\scshape}{}{8pt}{}
\theoremstyle{definizione}
\newtheorem{definition}[theorem]{Definition}

\def\i{\textbf i}
\def\Re{\text{Re}\,}
\def\Im{\text{Im}\,}

\begin{document}
\title{On complex power nonnegative matrices}

\author[ ]{Francesco Tudisco \and Valerio Cardinali \and Carmine Di Fiore\\ \tiny{Department of Mathematics, University of Rome ``Tor Vergata'', Via della Ricerca Scientifica, 00133 Rome, Italy}}


\maketitle

\begin{abstract}
 Power nonnegative matrices are defined as complex matrices having at least one nonnegative integer power. We exploit the possibility of deriving a Perron Frobenius-like theory for these matrices, obtaining three main results and drawing several consequences. We study, in particular, the relationships with the set of matrices having eventually nonnegative powers, the inverse of M-type matrices and the set of matrices whose columns (rows) sum up to one. 
\end{abstract}
\vspace{10pt}

\section{Introduction}
Given a complex $n\times n$ matrix $A$, we call it \textit{power nonnegative} if  there exists an integer $k\geq 1$ such that $A^k$ is a nonnegative matrix. Spectral properties of real power positive matrices were investigated for instance in \cite{brauer, johnson-tarazaga, seneta-introduction-nonnegative}. Complex matrices whose powers $A^k$ are nonnegative (positive) for all $k$ large enough are called \textit{eventually nonnegative (positive)}. Real eventually nonnegative (positive) matrices  were introduced by Friedland \cite{friedland-eventually}. Such matrices have been widely studied and, in particular, several recent works aimed at extending some classical results of the Perron-Frobenius theory for nonnegative matrices, to eventually nonnegative matrices,  see for instance \cite{mcdonald2, mcdonald3, noutsos-eventually, noutsos-varga, tarazaga, mcdonald1}.  Is it possible to do the same for the more general power nonnegative matrices? Are power nonnegative and eventually nonnegative matrices related somehow? We investigate these problems alongside Section \ref{sec:power} obtaining the results in Theorems \ref{pow-nonneg}, \ref{serie}, \ref{pow-pos} and in several corollaries (see for instance Corollaries \ref{pow-eve-nonneg} and \ref{pow-even-pos}). Moreover we provide a new direct and selfcontained proof of the theorem concerning the Perron-Frobenius properties of an eventually nonnegative matrix \cite{noutsos-varga}. 

\subsection{Notations and preliminaries}
Any matrix is assumed to be a square complex matrix of order $n$, unless otherwise
specified. Given a matrix $M$ let $a_M(\lambda)$ and $g_M(\lambda)$ denote
the algebraic and geometric multiplicities of $\lambda$ as an eigenvalue of $M$,
respectively, and let $\sigma(M)$ be the spectrum of $M$. The
square zero matrix is denoted by $O$. The imaginary unit is denoted by $\i$,  $e=(1,\dots,1)^\t$ is the vector of all ones and $e_i$ is the $i$-th canonical vector $(e_i)_k = \delta_{ik}$. A
nonnegative (positive) matrix $A=(a_{ij})_{ij}$ is a matrix such that $a_{ij}\geq 0$ $(a_{ij}>0)$,
$\forall i,j$. For such matrices we use the symbol $A\geq O$ $(A>O)$ underlying the
partial order $A\geq B$ $\sse$ $A-B\geq O$. 

By saying that a complex matrix  $A$ is a \textit{weakly stochastic} matrix 
we mean that $A^\t e = e$ (weakly column stochastic). So any weakly stochastic matrix having nonnegative entries is a stochastic matrix in the usual sense. If $Ae=e$ we say that $A$ is weakly row stochastic. If both $A$ and $A^\t$ are 
weakly stochastic then $A$ is said weakly doubly stochastic. 

A matrix $A$ is called reducible if there exists a permutation matrix
$P$ such that 
$$PAP^\t=\matrix{cc}{X & O \\ W & Y}$$
where the diagonal blocks are square matrices. A matrix is said to be irreducible if it is not reducible. 

We recall the Perron-Frobenius theorem for square nonnegative, nonnegative irreducible and primitive matrices, respectively,  collecting the results stated in \cite{bapat-nonnegative, fiedler-special-matrices, varga}.
\begin{theorem}[Perron-Frobenius]\label{pf}
 Let $A\geq O$ be a square matrix and let $\rho(A)$ be its spectral radius. Then
\begin{enumerate}
\item[1.] $\rho(A) \in \sigma(A)$
\item[2.] There exist $x,y \geq 0$ such that $Ax=\rho(A)x$, $y^\t A = \rho(A)y^\t$, with $x,y	\neq 0$
\end{enumerate}
If moreover $A\geq O$ is irreducible, then
\begin{itemize}
\item[3.] $\rho(A)\in \sigma(A)$ is simple and nonzero
\item[4.] The right and left eigenvectors $x,y$ in 2 are positive and, as a consequence of 3, unique up to a scalar multiple
\item[5.] There exists $p \geq 1$ such that $\rho(A)\exp\(\f{2\pi \i h}p\right)$, $1\leq h\leq p$, are simple eigenvalues of $A$ and there is no other eigenvalue of modulus $\rho(A)$. 
\item[6.] There exists a permutation matrix $P$ such that
\begin{equation}\label{normale}
PAP^\t = \matrix{cccc}{ & A_1 & & \\  &  & \ddots & \\ & & & A_{p-1} \\ A_p & & &}
\end{equation}
where $p$ is as in 5, the diagonal blocks are zero square matrices, and the $A_i$,  are in general rectangular matrices.
\item[7.] $\f 1 {1+k}\sum_{s=0}^k\(\f A {\rho(A)}\right)^s \conv{k \to \infty} \f{xy^\t}{x^\t y}>O$
\end{itemize} 
and the following statements are equivalent
\begin{itemize}
\item[8.] $A$ is primitive, i.e. $\exists k >0$ such that $A^k >O$
\item[9.] There exists $k_0$ such that $A^k >O$ for all $k \geq k_0$
\item[10.] The integer $p$ in 5 is one
\item[11.] $\(\f A {\rho(A)}\right)^k \conv{k \to \infty} \f{xy^\t}{x^\t y}>O$.
\end{itemize}
\end{theorem}
Note that primitive matrices, which are sometimes called acyclic or aperiodic matrices, are by point 8 a subclass of eventually positive matrices \cite{varga, noutsos-eventually}. Moreover, if $A$ is primitive, the first $k$ for which $A^k$ is positive is usually called \textit{the exponent} of $A$, and is denoted by $\gamma(A)$.


\section{Power nonnegative matrices}\label{sec:power}
Given any $n\times n$ matrix $A$, let us denote by $\lambda_1(A)$,  $\lambda_2(A)$, $\dots$, $\lambda_s(A)$  its $s$ distinct eigenvalues,
with the convention that 
$$\rho(A)=|\lambda_1(A)|\geq |\lambda_2(A)|\geq \dots \geq
|\lambda_s(A)|$$

We give the following definition, which is a slight modification of the terminology introduced in \cite{brauer} and \cite{seneta-introduction-nonnegative}
\begin{definition}
 A square matrix $A$ is said to be \textit{power nonnegative (positive)} if $\exists k \geq 1$ such that $A^k\geq O$ $(>O)$. For such a matrix we let 
$$\nu(A) = \min\{k\geq 1 \mid A^k\geq O\} \qquad \pi(A)=\min\{k\geq 1 \mid A^k >O\}$$
be its nonnegative and positive exponent, respectively.
\end{definition}
Note that for a power positive matrix $A$ it holds $1\leq \nu(A)\leq \pi(A)$. Also,  $A$ is nonnegative if and only if $\nu(A)=1$, and $A$ is primitive if and only if $1=\nu(A)\leq \pi(A)=\gamma(A)$. 

A complex matrix $A$ such that $A^k\geq O$ $(>O)$ for any large enough $k$ is called \textit{eventually nonnegative (positive)}. Real eventually nonnegative (positive) matrices were introduced by Friedland in \cite{friedland-eventually} and have been widely studied, see for instance \cite{handelman, mcdonald3, tarazaga}. In particular Naqvi, McDonald, Noutsos, Varga and Zaslavsky showed in \cite{mcdonald2, noutsos-eventually, noutsos-varga, mcdonald1}, respectively,  that eventually nonnegative (positive) matrices maintain several of the Perron-Frobenius properties of nonnegative (positive) matrices. 

In spite of what one could suppose at a first glance, an eventually nonnegative matrix $A$ has not to be real, unless $A$ is nonsingular (in the latter case, in fact, the equality
$ ( A^r - \overline{A^r} ) A^s = 0$, which holds for all $r+s$
large enough, implies $A^r - \overline{A^r}=O$ for all $r$). Indeed if $A =U+\i V$ where $VU=UV=O$, $V$ is real nilpotent and $U$ is real eventually nonnegative, then $A$ is a (purely) complex eventually nonnegative matrix. 
For example 
\begin{equation}\label{esempio:carlo}
A = U + \i V, \quad U = \matrix{ccc}{1 & 1 & 1 \\ 1 & 1 & 1\\ 1 & 1 & 1}, \quad V= \matrix{ccc}{1 & 1  & -2 \\ -1 & -1 & 2  \\  0 & 0  &  0}
\end{equation}
Actually Zaslavsky and Tam have shown in \cite{complex-eventually} that this is somehow the only possibility. In fact they observed that any complex matrix $A$ can be uniquely represented as $B_A+N_A$ where $N_A$ is a nilpotent matrix such that $B_AN_A=N_AB_A=O$ and $B_A$ is a matrix whose singular elementary Jordan blocks\footnote{That is an elementary Jordan block relative to a zero eigenvalue.} (if any) are $1$-dimensional. Thus they showed that $A$ is eventually nonnegative (positive) if and only if $B_A$ is a real eventually nonnegative (positive) matrix.

Eventually nonnegative (positive) matrices are of course a proper subset of power nonnegative (positive) matrices. In what follows we derive  a Perron-Frobenius-like theory for power nonnegative (positive) matrices (Section \ref{sec:main-thm} and Theorems \ref{pow-nonneg}, \ref{serie}, \ref{pow-pos} therein).  From this we draw in Section \ref{sec:cons-thm} several consequences, obtaining some new and some known properties of real power positive matrices (see f.i. \cite{brauer, johnson-tarazaga, seneta-introduction-nonnegative}), extending Theorem 4.2 in  \cite{mcdonald-M-mx} to complex matrices and showing how complex eventually and power nonnegative matrices are related (Corollaries \ref{pow-eve-nonneg} and \ref{pow-even-pos} above all).
\section{Main results}\label{sec:main-thm}
Next  Theorems \ref{pow-nonneg}, \ref{serie} and \ref{pow-pos} give a  generalization of the Perron-Frobenius Theorem \ref{pf} to power nonnegative matrices. First of all observe that, if $A$ is power nonnegative, even in case $A$ is real, $\lambda_1(A)$ s.t. $|\lambda_1(A)|=\rho(A)$ and its right and left eigenvectors might be not real. In fact  
$$A = \matrix{cc}{0 & 1 \\ -1 & 0}$$
is such that $A^2=-I$, $A^3=A^\t$ and $A^4 = I \geq O$, thus $A$ is power nonnegative with $\nu(A)=4$. However the eigenspace of $\lambda_1(A)=\i$ (or, equivalently, $\lambda_1(A)=-\i$), does not contain real vectors, so properties 1 and 2 of Theorem \ref{pf} can not hold. Also note that such  $A$ is an example of power nonnegative matrix which is not eventually nonnegative. More in general one can easily propose examples of power nonnegative matrices which are not eventually nonnegative simply by considering idempotent matrices $A$ which are not nonnegative themselves.

\begin{theorem}\label{pow-nonneg}
 Let $A$ be a $n\times n$ power nonnegative matrix, and let $x, y \in \CC^n$ be such that $Ax=\lambda_1(A)x$, $y^\t A = \lambda_1(A)y^\t$. Then
\begin{enumerate}
\item[(i)] There exists an integer $h$, $1\leq h \leq \nu(A)$, such that $\lambda_1(A)=\rho(A)\exp\( \f {2\pi \i h}{\nu(A)}\right)$. 
\item[ ]\hspace{-32pt} If moreover $A^{k}$ is nonnegative and irreducible for some $k\geq \nu(A)$, then
\item[(ii)] $\lambda_1(A)$ is simple and nonzero, and $\rho(A)\in\sigma(A)$ if and only if $\lambda_1(A)=\rho(A)$.
\item[(iii)] The right and the left eigenvectors $x,y$ can be chosen positive and unique up to a scalar multiple.
\item[(iv)] If $\nu(A)$ and $k$ are coprime\footnote{We say that two integer numbers $a$ and $b$ are coprime if $\gcd(a,b)=1$.} then $\rho(A)\in\sigma(A)$. If $\rho(A)\notin\sigma(A)$ and $\nu(A)$ or $k$ are prime, then $\nu(A)$ is the least positive integer for which $A^{\nu(A)}$ is nonnegative and irreducible. 
\item[(v)] $\lambda_1(A)^{-1}A$ is similar to a power nonnegative weakly column stochastic matrix $C$ and to a power nonnegative weakly row stochastic matrix $R$, both having the same pattern of $A$ and such that $\nu(C)=\nu(R)=\nu(A)$.
\item[(vi)] If $\lambda \in \sigma(A)$ is such that $|\lambda|=\rho(A)$, then $\lambda = \rho(A)\exp\(\f {2\pi \i h}{pk}\right)$, where $p$ is the cardinality of $\{\lambda\in \sigma(A)\mid |\lambda|=\rho(A)\}$, and $1\leq h\leq pk$.
\item[(vii)] If $(A^{k})_{ii}>0$ for some $i$, then $|\lambda_1(A)|>|\lambda_2(A)|$.
\end{enumerate}
\end{theorem}
\begin{proof}(i) The Perron-Frobenius (PF) theorem applied to $A^{\nu(A)}$ implies that $\lambda_1(A^{\nu(A)})=\rho(A)^{\nu(A)}$ thus $\lambda_1(A)=\rho(A)\exp(2\pi \i h /\nu(A))$ for some $1\leq h\leq \nu(A)$.



(ii), (iii) Let $\mu$ be an eigenvalue of $A$ such that $\mu^k=\rho(A^k)=\rho(A)^k$. Since $A^k\geq O$ is irreducible,  the PF theorem applied to $A^{k}$ implies that $\mu^k$ is simple, nonzero, and has a positive right eigenvector $x>0$. Therefore its Jordan space is one dimensional, that is the Jordan canonical form of $A^k$ can be written as the direct sum $\mu^k \oplus J^\prime$ where $J^\prime$ is the part of the Jordan canonical form of $A^k$ relative to the eigenvalues  belonging to $\sigma(A^k)\setminus\{\mu^k\}$. It follows that $\mu$ appears in the Jordan decomposition of $A$ in a similar way, i.e. $x$ is an eigenvector of $A$ corresponding to $\mu$, it is positive and unique up to a scalar multiple. This proves (iii). To prove (ii) assume that two eigenvalues of $A$, say $\lambda_1$ and $\lambda_2$, are such that $\mu^k =\lambda_1^k=\lambda_2^k=\rho(A)^k$. Since $\rho(A)^k$ is simple, we have that  $\mu=\lambda_1=\lambda_2$ and $\mu$ is simple. Moreover    
since $Ax=\mu x $ then $A^{\nu(A)}x = \mu^{\nu(A)}x$. Now $A^{\nu(A)}\geq O$ and $x>0$ imply $\mu^{\nu(A)}>0$, i.e. $\mu^{\nu(A)}=\rho(A)^{\nu(A)}=\lambda_1(A)^{\nu(A)}$. Thus $\lambda_1(A)$ is the only eigenvalue such that $\lambda_1(A)^k=\rho(A)^k$, and this implies that $\rho(A)\in\sigma(A)$ if and only if $\lambda_1(A)=\rho(A)$. 

(iv) Assume that $\rho(A)\notin \sigma(A)$ and let us show that $\gcd(\nu(A),k)$ is not~$1$. By (ii) we see that $\lambda_1(A)=\rho(A)\exp(2\pi\i h/\nu(A))\neq \rho(A)$, hence $1\leq \gcd(\nu(A),h)=\nu'<\nu(A)$. Therefore there exists $q>1$ such that $\nu(A)=\nu'q$ and $h/\nu(A)=s/q$ for $1<s\leq q$. Since   $\lambda_1(A)^{\nu(A)}=\rho(A)^{\nu(A)}$ and $\lambda_1(A)^k=\rho(A)^k$, we have
$$e^{\i \f{2\pi h}{\nu(A)}}=e^{\i \f{2\pi m}{k}}=e^{\i \f{2\pi s}{q}}$$
for some $m \in \{1,\dots,k-1\}$.  Hence $k=k'q$ and $\gcd(\nu(A),k)\geq q>1$. Finally if either $\nu(A)$ or $k$ are prime, then $k$ is a multiple of $\nu(A)$, thus $A^{\nu(A)}$ is irreducible. 


(v) Let $D_x = \diag(x)$, $x$ being a positive right eigenvector relative to $\lambda_1(A)$. Then $\sum_{j}x_j a_{ij}x_i^{-1}=\lambda_1(A)$ for all $i=1,\dots,n$. Hence $R=\lambda_1(A)^{-1}D_x^{-1} A D_x$ is weakly row stochastic and has the same pattern of $A$; also $(D_x)_{ii} >0$ for any $i$, implies that $R^{\nu(A)}=D_x^{-1} \(\lambda_1(A)^{-1}A\right)^{\nu(A)}D_x$ is nonnegative and irreducible and $\nu(R)\leq \nu(A)$. Viceversa we see that $A^{\nu(R)}= (\lambda_1(A)D_xRD_x^{-1})^{\nu(R)}= D_x (\lambda_1(A)R)^{\nu(R)}D_x^{-1}\geq O$ that is $\nu(A)\leq \nu(R)$. 
Setting $D_y=\diag(y)$ one observes analogously that $C=\lambda_1(A)^{-1}D_y A D_y^{-1}$ is power nonnegative weakly column stochastic and that $\nu(C)\leq \nu(A)$. 

(vi) Let $\mu_1,\dots,\mu_p$ be the eigenvalues of $A^{k}$ of modulus $\rho(A)^{k}$. Since $A^{k}\geq O$ is irreducible, the PF theorem implies that $\mu_j=\rho(A)^{k}\exp\(2\pi \i h_j/p\right)$, $1\leq h_j \leq p$, $j=1,\dots,p$. The thesis now follows by observing that $p$ coincides with the number of eigenvalues of $A$ with modulus $\rho(A)$ and recalling that any eigenvalue of $A$ is a $k$-th  root of an eigenvalue of $A^k$. 

(vii) Since $A^k$ is nonnegative, irreducible and $(A^k)_{ii}>0$, then $A^k$ is a primitive matrix. The thesis follows.
\end{proof}

It is straightforward to observe that $A$ reducible implies $A^k$ reducible, for all $k$. Thus, given any matrix $A$, if there exists an integer $k$ such that  $A^k$ is irreducible, then $A$ must be irreducible itself. In particular any positive integer power of any primitive matrix must be irreducible \cite[Thm. 1.8.2]{bapat-nonnegative}. Actually we observe that an analogous property holds for any irreducible nonnegative matrix:
\begin{theorem}\label{irr}Let $A$ be a square nonnegative matrix. Then $A$ is irreducible if and only if there exists a divergent subsequence $(a_m)_m \subset \NN$ such that $A^{a_m}$ is irreducible for any $m=1,2,\dots$ In particular, $A$ is primitive if and only if $A^{a_m}$ is irreducible with $(a_m)_m = \NN$.
\end{theorem} 
\begin{proof}
Assume that $A$ is not primitive. If $A^{a_m}$ is irreducible then of course $A$ is irreducible. So let $A$  be irreducible. Let $P$ be the permutation matrix which transforms $A$ into the matrix $B=PAP^\t $ in \eqref{normale}. Now let $G_B = (V,E)$ be the graph associated with $B$, and observe that since $B$ is irreducible there exists a cycle $c(v)$ for any $v \in V$. Also, due to the structure of $B$, the length of $c(v)$ is a multiple of $p$, say $p \, s_v$. Now let $s$ be the least common multiple of such $s_v$,  $s=\mathrm{lcm}\{s_v \mid v \in V\}$. Then the graph associated with $B^{ps}$ contains all the loops, therefore $\exists \alpha >0$ such that $B^{ps}\geq \alpha I$. As a consequence $B^{ps+1}\geq \alpha B$ and $B^{msp}B^{sp+1}\geq \alpha^{m+1}B$. Therefore the irreducibility of $B$ implies the irreducibility of $B^{a_m}$, for $a_m = msp+1$, $m=0,1,2,\dots$, and hence the thesis. If $A$ is primitive, we refer the reader to  \cite[Thm. 1.8.2]{bapat-nonnegative} for a  proof of the statement.
\end{proof}
Observe that the assumption $A$ nonnegative in the theorem above is crucial. In fact the thesis no longer holds if $A$ is a generic irreducible matrix. In this latter case, indeed, the spectral radius $\rho(A)$ might be zero implying that $A^m=O$ for any $m$ large enough. As an example, consider a positive vector $u$,  since any vector of the form $u_j e_i - u_i e_j$, $i\neq j$, belongs to $\Span(u)^\bot$, there exists $v \in \Span(u)^\bot$ with no zero entries. Then the rank one matrix $A = vu^\t$ is irreducible, but $A^k=O$ is reducible for any $k\geq 2$. 
Nonetheless Theorem \ref{irr} fails to be valid, without the hypothesis $A\geq O$, not only for nilpotent matrices. In fact there exist matrices $A$ which are irreducible and non nilpotent but such that $A^m$ is reducible for all $m$ large enough. As an example consider the matrix 
$$ 
A=\matrix{cc}{vu^\t &  vv^\t \\ uu^\t & uv^\t}= \matrix{cc}{ & \,\,\,\,\,\, \\ uu^\t &}+\matrix{cc}{vu^\t &  vv^\t \\  & uv^\t}=A_1+A_2
$$
By definition $A$ is irreducible, $A_1^2=O$, and $A_1A_2=A_2A_1=O$, then $A^m = A_2^m$ is reducible for all $m\geq 2$.

Let us note furthermore that it may happen that a power nonnegative matrix $A$ is irreducible, there exists a $k>\nu(A)$ such that $A^k$ is nonnegative and irreducible, but $A^{\nu(A)}$ is reducible. This fact is shown by the following example and we deduce that the hypothesis on the exponents in Theorem \ref{pow-nonneg} is sharp in this sense.
\begin{example}\label{AAA}
Consider the block matrix 
$$A =\matrix{cc}{ & B \\ C & }, \qquad B = \matrix{cc}{1 & 1 \\ 1 & 1}, \quad C =\matrix{cc}{1 & -x \\ 0 & 2}$$
It is easy to see that 
$$A^2 = \matrix{cc}{ BC &   \\ &  CB  }\qquad \text{and} \qquad A^3=\matrix{cc}{ & BCB \\ CBC & }$$
where 
\begin{gather*}
BC = \matrix{cc}{ 1& 2-x \\ 1 & 2-x }, \quad CB = \matrix{cc}{1-x & 1-x \\ 2 & 2}\\
BCB = \matrix{cc}{3-x & 3-x \\ 3-x & 3-x }, \quad CBC =\matrix{cc}{1-x & (1-x)(2-x) \\ 2 & 2(2-x) }
\end{gather*}
therefore, for any $x \in(0,1)$, $A$ is irreducible and power nonnegative with $\nu(A)=2$. Also, $A^{\nu(A)}$ is reducible whereas $A^k$, $k=3$, is nonnegative and irreducible. From Theorem \ref{pow-nonneg} (iv) it follows that $\rho(A)\in\sigma(A)$ since $\nu(A)$ and $k$ are both prime numbers. Indeed it is not difficult to observe that $\sigma(A)=\{\pm \sqrt{3-x},0\}$.
\end{example}

The hypothesis $A^{k}\geq O$ irreducible for some $k\geq \nu(A)$, implies many  properties on a power nonnegative matrix and we have already noted that under this assumption the matrix $A$ must be irreducible itself. One may therefore conjecture that some of the statements (ii)-(vii) of Theorem \ref{pow-nonneg} still hold under the weaker assumption that $A$ is irreducible and power nonnegative. Unlikely this is not the case. The following example shows, for instance, that property (ii) does not hold anymore. 
\begin{example}
Consider a nonnegative reducible matrix of the form
$$X = \matrix{cc}{X_1 & O \\ X_2 & X_1}$$
where $\rho(X_1)=\rho(X)\in \sigma(X)$. Assume that any $X_i$ is $n\times n$ symmetric, irreducible and $\rank(X_i)\leq n-2$.  Then $\ker X_1 \cap \ker X_2$ contains at least one nonzero vector $y$ laying outside the cone of nonnegative vectors. Let
$$Y = \matrix{cc}{O & y y^\t \\ O & O}.$$
By definition we have $Y^2=XY=YX=O$, thus the matrix $A=X+Y$ is irreducible and power nonnegative. One  easily observes that $\nu(A)=2$, so $\rho(A)$ or $-\rho(A)$ are eigenvalues of $A$. In particular we see that $\rho(A) \in \sigma(A)$ but its algebraic multiplicity is two. To this end let $\phi_M$ be the characteristic polynomial of $M$. Since for any invertible matrix $Q$ it holds
$$\matrix{cc}{Q & R \\ S & T} = \matrix{cc}{Q & O \\ S & I}\matrix{cc}{I & Q^{-1}R \\ O & T-SQ^{-1}R }$$ 
then, for any complex $\mu$ such that $|\mu|>\rho(A)=\rho(X)=\rho(X_1)$, 
\begin{align*}
\phi_A(\mu )&=\det(\mu I-A)=\det(\mu I-X_1)\det(\mu I-X_1-X_2(\mu I-X_1)^{-1}yy^\t)\\
&= \textstyle{\det(\mu I-X_1)\det(\mu I-X_1-X_2(\sum_{k\geq 0}\mu^{-(1+k)}X_1^k)yy^\t)}\\
&=\det(\mu I-X_1)^2=\phi_X(\mu)
\end{align*} 
Therefore the characteristic polynomials of $A$ and $X$ coincide and $a_A(\rho(A))=a_X(\rho(X))=2$, proving our claim.
\end{example}
Many authors have looked at extending some combinatorial properties of nonnegative matrices to eventually nonnegative matrices, observing that often the relationship between the combinatorial, Jordan and spectral structures of eventually nonnegative matrices is not consistent with that of nonnegative matrices, see for instance  \cite{friedland-eventually, mcdonald2, mcdonald1} and the references therein. In their investigations some examples analogous to the one we gave above and many others have been proposed, showing for instance that also statement (iii) of Theorem \ref{pow-nonneg} is no longer ensured without the request $A^{\nu(A)}$ irreducible. 

The next two lemmas, which are valuable in themselves, let us prove an interesting limit property for power nonnegative matrices  
\begin{lemma}\label{bound}
Let $A$ be power nonnegative such that $A^{m}$ is nonnegative and irreducible for some $m \geq \nu(A)$. For any $y \in \CC^n$  we have $\lim_{k\to \infty}\f 1 k \lambda_1(A)^{-k}A^k y~=~0$.
\end{lemma}
\begin{proof}
Let $B=\lambda_1(A)^{-1}A$. Theorem \ref{pow-nonneg} implies that there exists $x>0$ such that $Bx = x$. As a consequence $B^{s}x=x$ and $B^{s\nu(B)}\geq O$, for any positive integer $s$. Therefore $\max_\ell  x_\ell \geq x_i = \sum_{j=1}^n (B^{s\nu(B)})_{ij}x_j \geq \min_\ell x_\ell \sum_{j=1}^n (B^{s\nu(B)})_{ij}$
which implies that 
$$0 \leq (B^{s\nu(B)})_{ij}\leq \f {\max_i x_i}{\min_i x_i}, \qquad \forall s \in \NN$$
i.e the entries of $B^{s\nu(B)}$ are uniformly bounded. For any $k$ let $p \geq 0$ and $0\leq q<\nu(B)$ be such that $ k=p\nu(B)+q$. For any $y \in \CC^n$ we have
\begin{equation}\label{7}
\f 1 {(p+1)\nu(B)}|B^{p\nu(B)}(B^q y)|\leq \f 1 k |B^k y| \leq \f 1 {p\nu(B)}|B^{p\nu(B)}(B^q y)|
\end{equation}
Now since $B^{p\nu(B)}$ is entrywise bounded, both left and right hand sides of \eqref{7} converge to $0$ as $k$ diverges, concluding the proof.
\end{proof}
\begin{lemma}\label{directsum}
Let $A$ be as in the previous lemma and let $x>0$ be the right eigenvector relative to $\lambda_1(A)$. Then $\CC^n = \Span(x)\oplus (\lambda_1(A)I-A)\CC^n$
\end{lemma}
\begin{proof}
Let $B = \lambda_1(A)^{-1}A$. Then $Bx=x$ and $(\lambda_1(A)I-A)\CC^n = (I-B)\CC^n$. Therefore the thesis follows if we prove that $\CC^n = \Span(x)\oplus (I-B)\CC^n$. To this end let $B_k=\f 1 {1+k} \sum_{i=0}^k B^i$, and let $p\geq 0$ and $0\leq q<\nu(B)$ be such that $k=p\nu(B)+q$, then
$$(1+k)B_k = I+\Bigl(\sum_{i=0}^{p-1} B^{i\nu(B)}\Bigr)\Bigl(\sum_{j=1}^{\nu(B)}B^j\Bigr)+B^{p\nu(B)}\sum_{j=1}^q B^j\, .$$
Now arguing as in Lemma \ref{bound}, we observe the following inequalities
\begin{itemize}
	\item[ ] $B^{s\nu(B)}\leq \left(\f{\max_i x_i}{\min_i x_i}\right),$ thus $\left|\left(\sum_{s=0}^{p-1}B^{s\nu(B)}\right)_{ij}\right|\leq p\left(\f{\max_i x_i}{\min_i x_i}\right)$ 
	\item[ ]$\ba{lll}
	\left|\left(\sum_{s=0}^{m}B^s\right)_{ij}\right| &\leq& \left(\f{1-n^{m+1}}{1-n}\right) \max_{t=0,\dots,m}(\max_{ij}|b_{ij}|)^t\\
	&\leq & \left(\f{1-n^{m+1}}{1-n}\right)(1+\max_{ij}|b_{ij}|)^m\ea$
\end{itemize}
which combined with the fact that $|(MQ)_{ij}|\leq n (\max_{ij}|m_{ij}|)(\max_{ij}|q_{ij}|)$ for any two matrices $M,Q$, let us obtain the following bound, holding for any $y \in \CC^n$ 
$$|(B_ky)_i| \leq \f {n^2}{\nu(B)}(\max_i |y_i|)(1+\max_{ij}|b_{ij}|)^{\nu(B)}\left(\f{1-n^{\nu(B)+1}}{1-n}\right)\left(\f{\max_i x_i}{\min_i x_i}\right) + O\left(\f 1 k \right)$$
Therefore for any vector $y \in \CC^n$ the sequence $y_k = B_k y$ is entrywise bounded, thus there exists a convergent subsequence $y_{k_j}$. Let $\tilde y \in \CC^n$ be its limit. We claim that $\tilde y \in \Span(x)$. In fact, due to Lemma \ref{bound}, $y_{k_j}-By_{k_j} = \f 1 {1+k_j}(I-B^{1+k_j})y \conv{j\to \infty}0$ hence $B\tilde y = \tilde y$ and Theorem \ref{pow-nonneg} implies $\tilde y \in \Span(x)$. Now given any vector $y \in \CC^n$ we write it as $y = B_{k_j} y +(I-B_{k_j})y$ and,  taking the limit, we get $y = \tilde y + (y-\tilde y)$. To conclude the proof we need to show that $y-\tilde y \in (I-B)\CC^n$. Since $(I-B)\CC^n$ is closed it is enough to show that $(I-B_{k_j})y \in (I-B)\CC^n$, and this is easily seen since $B=I$ is a root of the matrix polynomial $(I-B_k)$, i.e. for any $k$, $I-B_k = (I-B)f(B)$ for some polynomial $f$.
\end{proof}
Lemmas \ref{bound} and \ref{directsum} are the basis for the following
\begin{theorem}\label{serie}
Let $A$ be a power nonnegative matrix such that $A^{m}$ is nonnegative and irreducible for some $m \geq \nu(A)$. Let $x,y>0$ be such that $Ax=\lambda_1(A)x$ and $y^\t A = \lambda_1(A)y^\t $. Then 
$$\f 1 {1+k} \sum_{s=0}^k \lambda_1(A)^{-s}A^s\conv{k\to \infty}\f {xy^\t}{x^\t y}>O$$
\end{theorem}
\begin{proof}
Set $A_k = \f 1 {1+k} \sum_{s=0}^k \lambda_1(A)^{-s}A^s$. If $v \in (\lambda_1(A)I-A)\CC^n$ then there exists $z \in \CC^n$ such that $v = (\lambda_1(A)I-A)z$ and due to Lemma \ref{bound} $A_k v = \f 1 {1+k} (\lambda_1(A)I-\lambda_1(A)^{-k}A^{k+1})z\conv{k\to \infty}0$. On the other hand, any $v \in \Span(x)$ is a fixed point of $A_k$, thus $v = \lim_k A_k v$. Now by virtue of Lemma \ref{directsum} we can decompose $\CC^n$ as the direct sum $\Span(x)\oplus (\lambda_1(A)I-A)\CC^n$, thus for any $v \in \CC^n$ we get $\lim_{k\to \infty}A_k v \in \Span(x)$. Applying the same argument to $A^\t$ we see that, for any $v \in \CC^n$, $\lim_{k\to \infty}A_k^\t v \in \Span(y)$. Therefore $A_k$ converges punctually to the rank one matrix $\alpha xy^\t$. Finally since $A_k x = x$ for any $k$, we have $(\alpha xy^\t)x=x$ implying that $\alpha =\|xy^\t\|^{-1}$.
\end{proof}

As for the case of primitive matrices, we consider the case of power positive matrices separately, and observe that the following result holds
\begin{theorem}\label{pow-pos}
Let $A$ be a power positive matrix, then
  \begin{enumerate}
    \item[(i)] $\lambda_1(A)$ is simple nonzero and $|\lambda_2(A)|<|\lambda_1(A)|$. 
\item[(ii)] $\(\lambda_1(A)^{-1}A\right)^p \conv{p\to \infty}\f{xy^\t}{x^\t y}>O$, where $x$ and $y$ are right and left positive eigenvectors of $A$ corresponding to $\lambda_1(A)$, respectively.
      \item[(iii)] $\lambda_1(A)^{-1}A$ is similar to a weakly doubly stochastic matrix via a diagonal plus rank one similarity transform.
    \end{enumerate}
    Moreover (i) and (ii) are equivalent.
\end{theorem}
\begin{proof}
 (i) Since $A$ is power positive,  $A^{\nu(A)}$ is nonnegative and irreducible, then there exists $h$, $1\leq h\leq \nu(A)\leq \pi(A)$, such that $\lambda_1(A)=\rho(A)\exp\(2\pi \i h/\nu(A)\right)$ is simple and nonzero. Moreover $\rho(A^{\pi(A)})$ is a positive simple eigenvalue of $A^{\pi(A)}$ and the remaining eigenvalues of $A^{\pi(A)}$ have absolute value smaller than $\rho(A^{\pi(A)})$. Therefore $|\lambda_2(A)|<|\lambda_1(A)|$.

(i)$\then$(ii)  Set briefly $\lambda_i = \lambda_i(A)$, and let $A = XJX^{-1}$ be the Jordan decomposition of $A$, where $J = \lambda_1 \oplus J(\mu_2) \oplus \cdots \oplus J(\mu_s)$, $\mu_i$ are the distinct eigenvalues of $A$ except for $\lambda_1$, and $J(\mu_i)$ is the Jordan block relative to the $i$-the eigenvalue $\mu_i$. If  $i\neq 1$, then the  spectral radius of each matrix $\lambda_1^{-1}J(\mu_i)$ is smaller than $1$, hence $\lambda_1^{-p}J(\mu_i)^p$ converges to zero as $p$ diverges. As a consequence $$(\lambda_1^{-1}A)^p\conv{p\to \infty}Xe_1 e_1^\t X^{-1}=\f{xy^\t}{x^\t y}>O\, ,$$ 
where $x$ and $y$ are fixed positive right and left eigenvectors of $A$ corresponding to $\lambda_1$. 

(ii)$\then$(i) For simplicity let $B=\lambda_1(A)^{-1}A$ and $\B = xy^\t (x^\t y)^{-1}$. Since $B^p\conv{p\to \infty}\B$ then $a_B(\lambda_1(B))\leq a_{\B}(\lambda_1(\B))=1$. Now, for any $\ep>0$ there exists $p$ such that $\|B^p-\B\|<\ep$ therefore $|\lambda_i(B^p)-\lambda_i(\B)|<\ep$. In particular, since $\lambda_1(\B)=1$ and $\lambda_i(\B)=0$ for $i\geq 2$,  $|\lambda_1(B^p)-1|<\ep$ and $|\lambda_2(B^p)|<\ep$. For $\ep=1$ we get $|\lambda_2(B^p)|=|\lambda_2(B)|^p=|\lambda_2(A)|^p |\lambda_1(A)|^{-p}<1$ hence $|\lambda_1(A)|>|\lambda_2(A)|$.

(iii) Let $x$ and $y$ be as in (ii). Since $A^{\pi(A)}$ is irreducible, $\lambda_1(A)$ is nonzero, thus $y^\t x = \lambda_1(A)^{-1}y^\t A x$ is nonzero. Therefore we can assume w.l.o.g. that $y^\t x = 1$.  Consider the diagonal plus rank one matrix
\begin{equation}\label{S}
 S=D_y + (e-D_y x)y^\t
\end{equation}
where $D_y = \diag(y)$. Observe that $Sx = e$ and $S^\t e=ny$, therefore $(SAS^{-1})e=SAx=\lambda_1(A)e$ and $(SAS^{-1})^\t e =nS^{-\t}A^\t y=\lambda_1(A)e$. The thesis comes by multiplying by $\lambda_1(A)^{-1}$ the previous relations.
\end{proof}
\subsection{A proof of the Perron-Frobenius properties of complex eventually nonnegative matrices}
Theorem \ref{pow-nonneg} ensures that a power nonnegative matrix $A$ such that $A^k$ is nonnegative and irreducible for some $k\geq \nu(A)$ has a simple nonzero eigenvalue $\lambda_1(A)$ of maximum modulus whose right and left eigenvectors can be chosen positive. However in general $\lambda_1(A)\neq \rho(A)$. To observe this one can, for instance, consider a nonnegative and irreducible matrix $M$ and then let $A = e^{\i 2\pi h/k}M$. Obviously $A$ is power nonnegative, moreover $A^k$ is nonnegative and irreducible, but $\lambda_1(A)=\rho(A)e^{\i 2\pi h/k}\neq \rho(A)$.

However if $A$ is any complex
matrix such that $A^k$ is nonnegative for all large enough integer powers $k$ (i.e. $A$ is eventually nonnegative), then  $\rho(A)$ is an eigenvalue of $A$ with nonnegative right and left eigenvectors. This fact has been observed by Noutsos and Varga in \cite{noutsos-varga} and precisely it follows as a special case of Theorem 2.3 in that paper. Nevertheless in next Theorem \ref{rho-even} we propose a direct and self contained proof of this fact.

Also let us point out briefly here that it may happen that $\rho(A)$ is an eigenvalue of a power nonnegative matrix to which correspond nonnegative right and left eigenvectors, but $A$ is not eventually nonnegative. As an example consider a positive matrix $B$ with $\rho(B)>1$, and let 
$$A = \matrix{cc}{B & \\ & -I}$$
Such matrix $A$ is power nonnegative with $\nu(A)=2$,  $\rho(A)=\rho(B) \in \sigma(A)$ is simple and thus the nonnegative right and left eigenvectors of $A^{\nu(A)}$ are eigenvectors of $A$. However $(-I)^m\leq O$ for any odd power $m$ and thus $A$ is not eventually nonnegative. 

In order to prove Theorem \ref{rho-even} we need two preliminary results stated in
Lemmas \ref{lem:eventually1} and \ref{lem:eventually2} here below. For a complex number $z$ let $\arg(z) \in [0,2\pi)$ denote its argument, that is let $z = |z|e^{\i \arg(z)}$.
\begin{lemma}\label{lem:eventually1} For any matrix $A$ there exists an integer $k>1$ such that the $k$-th powers of distinct eigenvalues in $\sigma(A)$ are distinct eigenvalues in $\sigma(A^k)$.
\end{lemma} 
\begin{proof}
Let us show that there exists a $k$ such that for any  two eigenvalues $\lambda,\mu \in \sigma(A)$, if $\mu^k$ and $\lambda^k$ coincides in $\sigma(A^k)$ then $\mu = \lambda$.  Consider the set 
$$D_A = \left\{d\mid \arg(\lambda_1)-\arg(\lambda_2)=2\pi 	\f s d, \, \lambda_1, \lambda_2 \in \sigma(A),\, s,d \in \ZZ, \, \gcd(s,d)=1, \, d\neq 0 \right\}$$
Since $D_A$ is a finite set, there exists a prime number $k \notin D_A$. Assume that $\lambda^k = \mu^k$, for a  given pair $\lambda,\mu \in \sigma(A)$. Therefore two cases are possible: $\mu=\lambda$ or $\arg(\mu)-\arg(\lambda)=2\pi \f p k$ for an integer $p$. But of course this second case is not possible since $k \notin D_A$ is prime.
\end{proof}
\begin{lemma}\label{lem:eventually2} If $A^k \geq O$ for all  $k\geq k_0$ then $\lambda_1(A)=\rho(A)$.
\end{lemma} 
\begin{proof}
Since $A^k\geq O$, for all $k\geq k_0$ there exists $\mu \in \sigma(A)$ such that $\mu^k = \rho(A^k)=\rho(A)^k$. Therefore 
$|\mu| = \rho(A)$ and $ \arg(\mu) \in 2\pi  \QQ$. So, let $\mu_1, \dots, \mu_m$ be the eigenvalues of $A$ such that  $|\mu_i|=\rho(A)$ and $\arg(\mu_i) \in 2\pi \QQ$. Precisely, for any $i=1,\dots,m$ let $p_i,q_i$ be coprime integers such that $1\leq p_i \leq q_i$ and 
$$\arg(\mu_i)=2\pi \frac{p_i}{q_i}$$
Observe that if $\mu_i^k$ is real then only  two cases may happen
\begin{enumerate}
\item[$(i)$] $p_i=q_i$, i.e. $\mu_i$ is real
\item[ ]\hspace{-28pt} or
\item[$(ii)$] $p_i<q_i$ and $q_i$ divides $k$, i.e. any integer $k$ such that  $\mu_i^k = \rho(A)^k$ belongs to the set $q_i \NN$. 
\end{enumerate}
If $(i)$ holds, the thesis follows. Let us show that $(ii)$ can not happen for all $i=1,\dots,m$. In this case, indeed, the only possible exponents $k$ of $A^k$ for which $\rho(A^k)$ is an eigenvalue of $A^k$ are those belonging to $\cup_{i=1}^m q_i \NN$. The absurd now follows since $A^s\geq O$ for any $s \in \NN+k_0$ but of course $\cup_{i=1}^m q_i\NN$ can not entirely cover  $\NN+k_0$.
\end{proof}
\begin{theorem}\label{rho-even}
Let $A$ be eventually nonnegative, then $\rho(A) \in \sigma(A)$ and its right and left eigenvectors can be chosen nonnegative.
\end{theorem}
\begin{proof}
The fact that $\rho(A) \in \sigma(A)$ directly follows by Lemma \ref{lem:eventually2}. Due to Lemma \ref{lem:eventually1} there exists $k$ such that the $k$-powers of distinct eigenvalues in $\sigma(A)$ are distinct eigenvalues in $\sigma(A^k)$. Therefore the left and right eigenspaces of $A$ relative to $\rho(A)$ coincide with the left and right eigenspaces of $A^k$ relative to $\rho(A^k)$, respectively. Hence, since we can obviously assume  that $k\geq k_0$, the thesis follows because $A^k\geq O$ has left and right nonnegative eigenvectors relative to $\rho(A^k)$.
\end{proof}
%
\section{Consequences}\label{sec:cons-thm}
In this section we collect several relevant corollaries which more or less directly follow by our three main Theorems \ref{pow-nonneg}, \ref{serie} and \ref{pow-pos}. Some of them are known results proved previously by various authors.

\subsection{On the relation between power and eventually nonnegative matrices}
From the proof of (iv) in Theorem \ref{pow-nonneg} it follows that there exists a positive integer number $q\leq\nu(A)$ such that $k=k'q$, $\nu(A)=\nu' q$ (note that here we do not assume $q$ nontrivial), where $k$ is an integer power for which matrix $A^k$ is nonnegative and irreducible. Therefore $A^q$ is irreducible, being $A^k$ irreducible, and $\rho(A^q)$ is a simple nonzero eigenvalue of $A^q$, since $\lambda_1(A)=\rho(A)\exp( \f {2\pi \i h}{\nu(A)})=\rho(A)\exp( \f {2\pi \i s}{q})$.  Moreover its left and right eigenvectors (which thus are unique) can be chosen positive. Also note that whenever $A^m\geq O$ for an integer $m \geq \nu(A)$, $m$ has to be a multiple of $q$. Therefore the matrix $A^q$ has the same core properties that a generic nonnegative irreducible matrix has. Of course $A^q$ is a power nonnegative matrix and one could guess it to be eventually nonnegative or even nonnegative. This is in fact the case:
\begin{corollary}\label{pow-eve-nonneg}Let $A$ be power nonnegative such that $A^k$ is nonnegative and irreducible for some $k \geq \nu(A)$. If $q = \gcd(\nu(A),k)$, then $A^q$ is eventually nonnegative. 
\end{corollary}
\begin{proof}
Since $q=\gcd(\nu(A),k)$ there exist integers $\nu'$ and $k'$ such that  $\nu(A)=\nu^\prime q$, $k=k' q$ and $\gcd(\nu', k')=1$. Let $\F(\nu',k') = k'\nu'-k'-\nu'$ be the Frobenius number of $\nu'$ and $k'$ \cite{sylvester-numero-frobenius}. Since $\nu'$ and $k'$ are coprime then any $m > \F(\nu',k')$ can be written as $m= a \nu' + b k'$ for suitable nonnegative integers $a,b$. Thus $(A^q)^m=A^{a \nu(A)}A^{bk}\geq O$ for any $m> \F(\nu',k')$.
\end{proof}

Of course any eventually nonnegative matrix is power nonnegative, but a more noticeable relation is revealed by Corollary \ref{pow-eve-nonneg}. Note for instance that we can immediately conclude that the matrix $A$ of Example \ref{AAA} is eventually nonnegative.

Also a more strict connection between eventually positive and power positive matrices is shown here below in Corollary \ref{pow-even-pos}.

\begin{corollary}\label{pow-even-pos}
$A$ is power positive if and only if there exist integers $h,k$, such that $e^{2\pi \i h/k}A$ is eventually positive.
\end{corollary}
\begin{proof} Due to Theorem \ref{pow-pos}, if $A$ is power positive then $\lambda_1(A)=\rho(A)e^{\i \theta}$, $\theta=2\pi \i h / \pi(A)$. Moreover, by Theorem \ref{pow-pos} $(\lambda_1(A)^{-1}A)^k$ converges to a positive matrix (as $k$ diverges). Therefore, if $B=e^{-\i \theta}A$, then $\rho(B)=\rho(A)$ and $(\rho(A)^{-1}B)^k$ is positive for all $k$ large enough, implying the thesis. The reverse implication is obvious: since $B=e^{2\pi\i h/k }A$ is eventually positive, there exists $s$ such that $A^{sk}>O$.
\end{proof}
Observe that the same relation can not hold between power nonnegative and eventually nonnegative matrices. This is shown by the following simple example. Consider the  matrix
$$A = \matrix{cc}{ & I \\ -I &}$$
where the identity matrices are square matrices of the same order. One easily see that $\nu(A)=4$ and $A^{\nu(A)}=I$. Therefore, as already noted, $A$ can not be eventually nonnegative. Moreover $e^{i\theta}A$ is not eventually nonnegative for any $\theta \in \RR$. In fact, for any integer $p$, the nonzero entries of $(e^{\i \theta}A)^{p\nu(A)+1}$ are  $e^{\i \theta}$ and $-e^{\i \theta}$, which can not be both nonnegative numbers.

Let us recall for completeness that a converse version of Corollary \ref{pow-eve-nonneg} follows by inspecting the proof of \cite[Thm. 3.4] {mcdonald2}. We state it here below:
\begin{theorem}[\cite{mcdonald2}, Thm. 3.4]\label{teo:mcdonald} Let $A$ be a real eventually nonnegative matrix. If $A$ is nonsingular, or zero is a simple eigenvalue of $A$, then there exists $k\geq 1$ such that $A^k$ is nonnegative and irreducible.
\end{theorem}
Note indeed that, by applying the previous theorem to $A^q$, we get: \textit{If there exists an integer $q$ such that $A^q$ is  eventually nonnegative, and if $A$ is nonnsingular or has zero as simple eigenvalue, then  there exists $k\geq q$ such that $A^k$ is nonnegative and irreducible.} 

The relation with Corollary \ref{pow-eve-nonneg} is made evident by this latter way of stating Theorem \ref{teo:mcdonald}.
\subsection{M-type matrices based on power nonnegative matrices}
Nonnegative matrices play a central role in the theory of M-matrices (see f.i. \cite{positive-book} or \cite{fiedler-special-matrices}). An M-matrix is a matrix of the form $M =  \sigma I - A$, where $A\geq O$  and $\sigma\geq \rho(A)$. One of the main properties of such matrices concern their inverse, it is known indeed that the inverse of a nonsingular M-matrix is nonnegative, whence is positive if, in addition, the M-matrix is irreducible. The same thing can not be said in the general case for a matrix of the form $\sigma I - A$, with $A$ power nonnegative. In \cite{mcdonald-M-mx} examples are shown in this direction. Nevertheless by virtue of Theorem \ref{serie} we obtain the following further result,  showing that if $|\sigma|$ is \textit{not too big} and $A$ has a nonnegative irreducible power, then the inverse of $I-\sigma^{-1}A$ has positive real part. 

To avoid ambiguities let us agree that the real and imaginary parts of a matrix $M = (m_{ij})_{ij}$ are hereafter denoted by $\Re M$ and $\Im M$, denoting the matrices $\Re(M)_{ij}=\Re(m_{ij})$ and $\Im(M)_{ij}=\Im(m_{ij})$, respectively.
\begin{theorem}\label{teo:M-mx}Let $A$ be a power nonnegative matrix such that $A^k\geq O$ is irreducible for some $k\geq \nu(A)$. There exists $\ep >0$ such that if $|\lambda_1(A)-\sigma|< \ep$ and $|\sigma|>|\lambda_1(A)|$ then $\Re ( I-\sigma^{-1}A)^{-1}$ is a positive matrix.
\end{theorem}
\begin{proof}
For ease of notation let $\lambda_1 = \lambda_1(A)$. Theorem \ref{serie} shows that for any $k$ large enough 
$$\sum_{m=0}^k \lambda_1^{-m}A^m$$
has positive real part. Now let $\phi$ be such that $\sigma - \lambda_1 = \ep_0 e^{\i \phi}$ and consider the matrix
$$ M_k = \sum_{m=0}^k \sigma^{-m}A^m $$
Since the entries of such matrix depend continuously on $\ep_0$, a standard  continuity argument shows that there exists a positive $\ep$ such that if $\ep_0 \in (0,\ep)$ then $\Re M_k$ is positive, for any $k$ large enough. Finally since $|\sigma|>|\lambda_1|$, then  $( I - \sigma^{-1}A)^{-1}=\lim_{k\to \infty}M_k$ and we conclude that $\Re ( I-\sigma^{-1}A)^{-1}>O$.
\end{proof}
We would point out that, combining \cite[Thm. 3.4]{mcdonald2}, Theorems \ref{rho-even} and \ref{teo:M-mx} above, it can be easily observed that 
\begin{corollary}[\cite{mcdonald-M-mx}, Thm. 4.2]
If $A$ is a real eventually nonnegative matrix such that either $A$ is nonnsingular or zero is a simple eigenvalue of $A$, then there exits $\lambda>\rho(A)$ such that if $\lambda>\sigma>\rho(A)$ then $(\sigma I -A)^{-1}>O$.
\end{corollary}
\subsection{The case of real matrices}
It immediately follows from point (ii) of Theorem \ref{pow-pos} that if $A$ is power positive, then
\begin{itemize}
\item $\exists k_0$ such that $\Re(\lambda_1(A)^{-k}A^k)>O$ for any $k\geq k_0$
\item $\forall \ep>0$ $\exists k_\ep$ such that $\|\Im(\lambda_1(A)^{-k}A^{k})\|<\ep$, $\forall k \geq k_\ep$
\end{itemize}  
It is clear therefore that when $A$ is real, also a converse direction of points (i) and (ii) in Theorem \ref{pow-pos} holds. We state this fact in detail in next Corollary \ref{known}, whose proof  may be easily derived  as a special case of our Theorem \ref{pow-pos}.  We want to point out that Corollary \ref{known} is a well known result, see for instance \cite[Lemma 2.1 and Thm. 2.2]{handelman} and \cite{seneta-introduction-nonnegative}.
\begin{corollary}\label{known}
If $A$ is a real power positive matrix, the following statements are equivalent:
\begin{itemize}
\item There exists  an odd integer $k \geq \pi(A)$ such that $A^k >O$
\item $\rho(A) \in \sigma(A)$, it dominates the modulus of the other eigenvalues, the right and left eigenvectors $x$, $y$ relative to $\rho(A)$ are unique and  can be chosen positive.
\end{itemize}
\end{corollary}
As a simple consequence of Corollary  \ref{pow-even-pos} we deduce Thm. 3 in \cite{brauer} and Thm. 1 in \cite{johnson-tarazaga}, herebelow stated as Corollaries \ref{x} and \ref{y}. Their proof is omitted.
\begin{corollary}[\cite{brauer}, Thm. 3]\label{x} A real matrix $A$ is power positive if and only if $A$ or $-A$ are  eventually positive. 
\end{corollary}
\begin{corollary}[\cite{johnson-tarazaga}, Thm. 1]\label{y}
A real matrix $A$ is power positive and $A^k>O$ for an odd integer $k\geq \pi(A)$ if and only if $A$ is eventually positive.
\end{corollary}
Finally note that a bit more than what is claimed in Corollary \ref{y} can be said. To our knowledge this latter consequence, proved in Corollary \ref{cor:odd-irrid} below, has not been observed before.
\begin{corollary}\label{cor:odd-irrid} If $A$ is a real power nonnegative matrix such that $A^k$ is nonnegative and irreducible for an odd integer $k\geq \nu(A)$, then $\rho(A) \in \sigma(A)$.
\end{corollary}
\begin{proof}
Since $A$ is real, both $\lambda_1(A)$ and $\bar{\lambda_1(A)}$ (its complex conjugate) belongs to $\sigma(A)$. If $\lambda_1(A)\neq \bar{\lambda_1(A)}$ then the multiplicity of $\rho(A^k)$  as an eigenvalue of $A^k$ is larger or equal than two. But this is not possible since $A^k\geq O$ is irreducible. Hence $\lambda_1(A)=\pm \rho(A)$. Finally the oddness of $k$ implies $\lambda_1(A)=\rho(A)$.
\end{proof}

\subsection{The case of weakly stochastic matrices}
We recall that, in our notation, a weakly stochastic matrix $A$ is a matrix such that $\sum_i a_{ij}=1$, for any $j=1,\dots,n$ (we do not require the nonnegativity). Hence $1$ is always an element of the spectrum of a weakly stochastic matrix. This fact let us  state the following 
\begin{corollary}
 Let $A$ be weakly stochastic and power nonnegative. Then $1 = \rho(A)\in \sigma(A)$. 
 Moreover if $A^{m}$ is nonnegative and irreducible for some $m \geq \nu(A)$, then properties (ii)-(vii) of Theorem \ref{pow-nonneg} hold for $y=e$.
\end{corollary}
\begin{proof}
 We only need to prove that $1=\rho(A)$. It is easy to observe that $A$ weakly stochastic implies $A^k$ weakly stochastic for any integer $k$. Let $\lambda \in \sigma(A)$, then
$$\textstyle{|\lambda|^{\nu(A)}=|\lambda^{\nu(A)}|\leq\rho(A^{\nu(A)})\leq \|A^{\nu(A)}\|_1 = \max_j \sum_i |(A^{\nu(A)})_{ij}|=1}$$
and as a consequence $\rho(A)\leq 1$. We conclude since $1 \in \sigma(A)$.
\end{proof}
By inspecting the proof of Theorem \ref{pow-nonneg} one notes that property (v) in it actually only requires that $A$ has a positive dominant eigenvector and that $\lambda_1(A)$ is nonzero. This fact is  stated by the next Corollary \ref{stoch-simil} where the nilpotent case is considered as well.

Observe that if $\rho(A)=0$, then $A$ is nilpotent and thus $A$ is power nonnegative. Moreover, obviously, if an integer $p$ is such that $A^p=O$ then $p \geq \nu(A)$. However, under the assumption that $A$ has a positive dominant eigenvector we see that $p=\nu(A)$. Namely, we have the following
\begin{lemma}\label{lem:nilpot}Let $A$ be power nonnegative such that $\rho(A)=\lambda_1(A)=0$ and such that $A$ has a positive dominant eigenvector (i.e. a positive eigenvector relative to zero). Then $A^{\nu(A)}=O$.
\end{lemma} 
\begin{proof}
Since $Ax=0$ for a positive vector $x$, then $A^{\nu(A)}x=0$. If indices $i,j$ exist such that $(A^{\nu(A)})_{ij}>0$ then  $(A^{\nu(A)}x)_i = \sum_{k}(A^{\nu(A)})_{ik}x_k\geq (A^{\nu(A)})_{ij}x_j>0$ which yields an absurd.
\end{proof}
\begin{corollary}\label{stoch-simil}
Consider a square matrix $A$. 
\begin{itemize}
\item If $\rho(A)>0$, then $A$ is power nonnegative, and there exists a positive dominant eigenvector of $A$ if and only if $\lambda_1(A)^{-1}A$ is similar to a weakly row stochastic power nonnegative matrix $S$ via positive definite diagonal similarity transform, and $\nu(S)= \nu(A)$.
\item If $\rho(A)=0$ then $A$ is power nonnegative and $A^{\nu(A)}=O$. Moreover there exists a positive dominant eigenvector of $A$ if and only if $A$ is similar via a positive definite diagonal similarity transform to a nilpotent matrix $L$ whose rows sum up to zero.
\end{itemize}  
\end{corollary}
\begin{proof}
 Let $\rho(A)>0$. Assume that $\lambda_1(A)^{-1}A$ is similar to a weakly stochastic power nonnegative matrix $S$ via a  positive definite diagonal similarity transform, that is there exists  a diagonal matrix $D$, $(D)_{ii}>0$, such that $\lambda(A)^{-1}A = DSD^{-1}$. Then $x = De>0$ is a positive eigenvector of  $A$, $\lambda_1(A)$ is its eigenvalue and $\lambda_1(A)^{-\nu(S)}A^{\nu(S)}=DS^{\nu(S)}D^{-1}\geq O$, thus $$\lambda_1(\lambda_1(A)^{-1}A)=\lambda_1(A)^{-1}\rho(A)e^{2\pi \i h /\nu(S)}=1\, .$$
This implies $\lambda_1(A)^{\nu(S)}> 0$ and $A^{\nu(S)}\geq O$, which combined with $S^{\nu(A)}=D(\lambda_1(A)^{-1}A)^{\nu(A)}D^{-1}\geq O$ implies $\nu(A)=\nu(S)$. The reverse implication can be proved analogously. Now let $\rho(A)=0$ and assume that $L = D^{-1}AD$ is such that $Le=0$, for a diagonal $D$ with $(D)_{ii}>0$. Then $x = De$ is positive and $Ax = DLD^{-1}De=0$. Since $\rho(A)=0$, from Lemma \ref{lem:nilpot} we get $A^{\nu(A)}=O$. Again, the reverse implication can be proved the same way.
\end{proof}

We conclude with an example of a class of power nonnegative matrices. 
\begin{example}
 Let $\theta \in \CC$  and consider the weakly doubly stochastic circulant matrix 
$$C(\theta) = \matrix{ccccc}{
{\tiny 1-(n-1)\theta}  & \theta        & \cdots     &  \theta       \\ 
\theta         & 1-(n-1)\theta & \ddots     &   \vdots      \\
 \vdots        &   \ddots      & \ddots     &  \theta             \\
\theta         & \cdots        &  \theta    &    1-(n-1)\theta}, $$
if $E=ee^\t$ we briefly write $C(\theta) = (1-n\theta)I+\theta E$. Since $E^p = n^{p-1}E$, we have 
\begin{align*}
 C(\theta)^k &= \bigl((1-n\theta)I+\theta E\bigr)^k = \sum_{i=0}^k \binom{k}{i}(1-n\theta)^{k-i}\theta^i E^i\\
     &= (1-n\theta)^k I + n^{-1}\(\sum_{i=1}^k \binom{k}{i}(1-n\theta)^{k-i}(n\theta)^i\)E\\
     &= (1-n\theta)^k I + \frac{1-(1-n\theta)^k}n E
\end{align*}
thus $C(\theta)^k = C(\theta_k)$ where $\theta_k = n^{-1}(1-(1-n\theta)^k)$ and $k \geq 1$. Now observe that $C(\theta_k)\geq O$ if and only if $0 \leq \theta_k \leq (n-1)^{-1}$ that is
\begin{equation}\label{theta}
 C(\theta_k)\geq O \iff -\f 1 {n-1} \leq (1-n\theta)^k \leq 1\, .
\end{equation}
It is now clear how to exhibit an example of power nonnegative matrix. Consider any integer  $s\geq 2$  and let $A_s$ be the matrix $A_s=C\(\f 1 n e^{\i \f \pi s}\right)$. Note that $A_s$ is not real. Now let us observe that, for any integer $m$, $A_s^{2ms}$ is real. In fact,
\begin{align*}
 \(1-  e^{\i \pi /s}\)^{2ms} &= \binom{2ms}{ms}(-1)^{ms}e^{\i {\pi m}} + \sum_{p = 0}^{ms-1}\binom{2ms}{p} (-1)^p \(e^{\i \f {\pi p}s} + e^{\i \f{\pi (2ms-p)}s} \)\\
&=  \binom{2ms}{ms} (-1)^{m(s+1)}+ 2\sum_{p = 0}^{ms-1}\binom{2ms}{p} (-1)^p \cos \(\f {\pi p}s\) \in \RR
\end{align*}
and thus $A_s^{2ms}=C\(\f {1-(1-e^{\i\pi /s})^{2ms}}n\right)$ is real. Then note that for any fixed $s\geq 3$ we have 
\begin{gather*}
	0<|1-e^{\i \f \pi s}|^2 = 2 (1-\cos \f \pi s )\leq 2(1-\cos \f \pi 3)=1 \Then	\\
	0 < |(1-e^{\i \f \pi s})^{2ms}|=|1-e^{\i \f \pi s}|^{2ms}\leq 1 \Then \\
	-1 \leq (1-e^{\i \f \pi s})^{2ms}\leq 1, \quad (1-e^{\i \f \pi s})^{2ms}\neq 0 \Then \\
	0 < (1-e^{\i \f \pi s})^{4ms}\leq 1
\end{gather*}
By \eqref{theta} and the inequality just obtained we see that the real matrix $A_s^{4m s} = (A_s^{2ms})^2 = C\(\f {1-(1-e^{\i\pi /s})^{4ms}}n\right)$ is nonnegative. So, for any $s\geq 3$, the matrix $A_s= C(\f 1 n e^{\i \f \pi s})$ is a non real matrix such that $A_s^{4ms}\geq O$ for all integers $m=1,2,3,\dots$
\end{example}

\section*{Acknowledgements}
The researches whose results are collected in this work have originated during the seminars we have held in the little hall 1103, now named Francesco S. De Blasi, of Tor Vergata math dept, also thanks to the active participation of professor Paolo Zellini and the master' students Barbara, Davide, Gianluca and Stefano. Special thanks goes to Carlo Pagano who brought the question of eventually nonnegative complex matrices to our attention, suggesting the example in \eqref{esempio:carlo}, and to Riccardo Fastella for his contribution in proving Theorem \ref{irr}.


\bibliographystyle{plain} 
\bibliography{bibliografiah}

\end{document}